\documentclass[a4paper, twoside,11pt]{amsart}

\usepackage[english]{babel}
\usepackage{latexsym}
\usepackage{amsfonts}
\usepackage{amstext}
\usepackage{amssymb}
\usepackage{amsmath}
\usepackage{amscd}
\usepackage[dvips]{graphicx}
\usepackage[dvips]{color}

\newtheorem{theorem}{Theorem}[section]

\newtheorem{lemma}{Lemma}[section]
\newtheorem{definition}{Definition}[section]

\newtheorem{remark}{Remark}[section]

\if\csname proof\endcsname\relax
\newenvironment{proof}{\noindent{\bf Proof.\,}}{\hfill $\Box$\par\medskip}
\fi

\newcommand{\rr}{\mathbb R}
\newcommand{\zz}{\mathbb Z}
\newcommand{\Id}{\mathop{\mathrm{Id}}\nolimits}

\title[On the classification of inner maps]
{On the classification of the regular components of inner maps.}
\author{I. Yu. Vlasenko}
\date{}

\begin{document}
\selectlanguage{english}

\begin{abstract}
The topological classification of inner maps on the fully
invariant regular components of a wandering set with a special
attracting boundary up to the topological conjugacy is defined  
in terms of distinguishing graph.
Two inner maps of the same degree are topologically equivalent on
their fully invariant regular components of a certain class 
if and only if their distinguishing graphs are equivalent.
\end{abstract}
\maketitle

\section{Introduction.}
Inner maps were introduced by Stoilov in \cite{Stoilov1938}.
Recall that a map is called open if the image of an open set is an
open set. A map is called discrete if a preimage of any point is a
discrete set (it consists of isolated points).
A continuous open discrete map is called inner map.
Note that inner maps on compact surfaces have finite number of preimages.

The most noticeable representatives of inner maps are
homeomorphisms and holomorphic maps.
In fact, an inner map of an oriented surface can be represented as
a composition of a non-constant analytical function and a homeomorphism
due to Stoilov theorem~\cite{Stoilov1938}.
It follows the class of inner maps is much wider than holomorphic
maps, for example, it includes all homeomorphisms.
While dynamics of homeomorphisms, diffeomorphisms and holomorphic maps
has been studied deeply in recent times 
the study of general inner maps with methods of dynamical systems
theory just makes its first steps.

The class of inner maps introduced here in
definition~\ref{def:pseudolinear} seems exotic at first 
glance, but this class has such renown representatives as
complex polynomials on the Riemannian sphere considered on the basin
of attraction to infinity due to the B\"{o}ttcher theorem \cite{Milnor1999}.
Classification of dynamics on the attraction basin is done up to the
topological conjugacy.

\begin{theorem}[Main theorem]\label{th:main_article4}
Restrictions of inner maps $f$, $g$ on the 
fully invariant regular components of wandering set 
with a fully invariant pseudolinear attracting isolated connected component
satisfying definitions~\ref{def:pseudolinear} and~\ref{def:genericity_barticle4}
are topologically equivalent if and only if
they have the same degree and the
corresponding distinguishing graphs are equivalent.  
\end{theorem}

\section{Preliminary information.}
Let $f\colon M\to M$ be an inner map of a compact surface M.
Denote by $O^+(x)=\{f^n(x)|n\ge 0\}$ a forward trajectory of $x$,
$O^-(x)=\{f^{-n}(x)|n\ge 0\}$ a backward trajectory of $x$,
and $O(x)=\{f^{-n}\circ f^m(x)|n,m\ge 0\}$ a full trajectory of $x$.
Let us call the set $O^\perp(x)=\{f^{-n}\circ f^n(x)|n\ge 0\}$ a neutral section of trajectory.

\begin{definition}
  A point $x$ is called wandering if there exists an open set $U(x)$
  such that $\forall n\in\zz, n\not=0$ $f^n\left(U(x)\right)\cap U(x)=\varnothing$.
\end{definition}

A point which is not wandering is called nonwandering. The set of
nonwandering points of $f$ is denoted by $\Omega(f)$.
Denote by $B_\varepsilon(X)$ the $\varepsilon$-neighborhood of X.

\begin{definition}
  A point $x$ is said to be $\omega$-regular wandering if it is wandering and $\forall \varepsilon>0$
  $\exists \delta>0$, $\exists N > 0$: $\forall n>N$
  $f^n\left(B_\delta(x)\right)\subset B_\varepsilon(\Omega)$.
\end{definition}

\begin{definition}
  A point $x$ is called $\alpha$-regular wandering if it is wandering and $\forall \varepsilon>0$
  $\exists \delta>0$, $\exists N > 0$: $\forall n>N$
  $f^{-n}\left(B_\delta(x)\right)\subset B_\varepsilon(\Omega)$.
\end{definition}

A point which is both $\alpha$-regular and $\omega$-regular is called regular.
A connected component of the set of regular wandering points is called
a regular component of wandering set.
A component $U$ is called periodic if there exists n such that $f^n(U)=U$.
When $n=1$ a periodic component $U$ is called invariant.
Also, when $f^{-1}(U)=U$, an invariant component is called fully invariant.
Fully invariant components are important because study of periodic 
components can be essentially reduced to the study of fully invariant
components.

\begin{definition}\label{def:attracting}
  An isolated connected component $\partial_0$ of the boundary
  $\partial U$ of a fully invariant regular component $U$ of wandering set 
  is attracting if
  it has a strictly invariant neighborhood $V(\partial_0)$
  (a neighborhood such that $\overline{f(V)}\subset V$).
\end{definition}

\begin{definition}\label{def:radially_pseudolinear}
  An attracting isolated connected component $\partial_0$ of the
  boundary $\partial U$ of a fully invariant regular component $U$ of
  wandering set is radially pseudolinear 
  if it has strictly invariant neighborhood $V(\partial_0)$ 
  such that there exists a regular compact invariant foliation of $U\cap V$. 
  Equivalently, it can be defined as a foliation by a 
  function $\Phi\colon V\to [0,1]$ such that $\Phi$ is regular in $U\cap V$ 
  and $\Phi(f(x))=\lambda\Phi(x)$, $\lambda>0$.
\end{definition}

\begin{definition}\label{def:pseudolinear}
  A radially pseudolinear attracting isolated connected component
  $\partial_0$ of the boundary $\partial U$ of a fully invariant
  regular component $U$ of wandering set is pseudolinear 
  if $\forall x \in U\cap V$ $O^\perp(x)$ is dense in the foliation
  fiber of $x$.
\end{definition}

\begin{definition}[Genericity condition.]\label{def:genericity_barticle4}
  The genericity condition is a
  condition such that every fiber of the $\Phi$ foliation has no more
  than one image of critical point.
\end{definition}

\begin{definition}
  A closed neighborhood $Q\subset U$ is called fundamental
  neighborhood of $U$ if 
  \begin{enumerate}
  \item it is closure of its interior;
  \item $\forall x\in U$ $O(x)\cap Q\not=\varnothing$ \\
    ($Q$ contains a representative of every full trajectory in $U$);
  \item if $x\in Q, x\not\in \partial Q$ then $f(x)\not\in Q$.
  \end{enumerate}
\end{definition}

\begin{definition}
  Fundamental neighborhood $Q$ is called saturated
  if for every $x\in Q$ $Q$ contains $O^\perp(x)$.
\end{definition}

\begin{lemma}\label{lm:fund_nei_is_ring}
Fundamental neighborhoods in the attracting neighborhood from the
definition~\ref{def:radially_pseudolinear} are homeomorphic to the ring.
\end{lemma}

\begin{proof}
Note that the attracting component of boundary is isolated, it
means that the strictly invariant neighborhood is connected, so the
foliation does not have holes. Otherwise other components of boundary 
will be attracted to the attracting component, which will contradict 
isolatedness.

Since the boundary of strictly invariant neighborhood should also
belong to the foliation and the foliation is regular compact then all fibers
are homeomorphic to the circle. Due to the regularity of the foliation,
the neighborhood does not contain critical points or their preimages.
Hence, the map between a fundamental neighborhood and its image is 
a regular covering. But a M\"ebius string can only cover a M\"ebius
string, handles can only cover handles, and if there exists any,
there should be infinitely many of them, as fundamental neighborhood
has infinitely many images. It contradicts to compactness.
Since a fundamental neighborhood can't have a handle or a M\"ebius
string, therefore $Q$ is homeomorphic to the ring. 
\end{proof}

\begin{remark}\label{rem:single_fiber}
  In the theorem\ref{th:main_article4} we consider a fully invariant
  pseudolinear attracting isolated connected component of the
  boundary. Because it is 
  fully invariant, then equiscalar lines of the function $\Phi$ from
  the definition~\ref{def:radially_pseudolinear} in the fundamental
  neighborhood $Q$ from the lemma~\ref{lm:fund_nei_is_ring} consist of
  a single circle with the property that for its each point
  the circle also contains the whole set $O^\perp$ of that point.
\end{remark}

\begin{lemma}
  A fully invariant regular component $U$ of wandering set with
  a radially pseudolinear attracting isolated connected component
  is either homeomorphic to the ring (have the only other component of
  the boundary) in case it has no critical points 
  or it has infinite number of boundary components.
\end{lemma}

\begin{figure}[htbp]
\begin{center}
\includegraphics[scale=0.6]{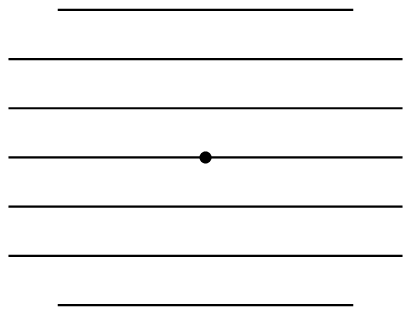}a)
\quad
\includegraphics[scale=0.6]{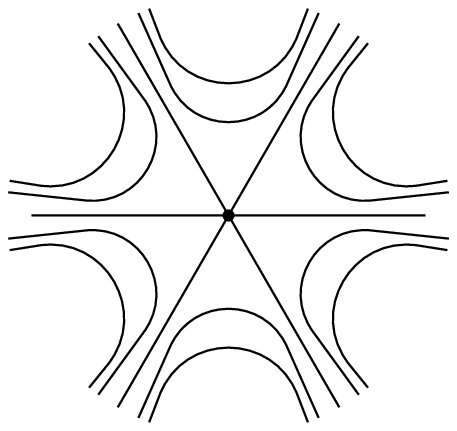}b)
\end{center}
\caption{a) Regular foliation b) foliation at the critical trajectory of
  degree 3.}\label{fig:foli-reg-sing}
\end{figure}

\begin{proof}
By lemma~\ref{lm:fund_nei_is_ring} all the fundamental
neighborhoods in the attracting neighborhood from the
definition~\ref{def:radially_pseudolinear} are homeomorphic to the ring.
In absence of critical points all the fundamental neighborhoods
are homeomorphic, hence the whole regular component is homeomorphic to
the ring.

In case when there are critical trajectories there are finite number of
them due to compactness. By definition~\ref{def:radially_pseudolinear},
the whole fully invariant neighborhood $U\cap V$ of the regular
component $U$ does not contain critical points.
Choose a fundamental neighborhood $Q\subset U\cap V$.  
Note that $Q$ is homeomorphic to the ring according to 
the lemma~\ref{lm:fund_nei_is_ring}. 

Consider the first preimage of $Q$ that contain a critical point.
It is enough to consider a single critical point case, the deduction 
in the case of multiple critical points is the same.

According to the Stoilov's ``lemma on simple curve''
\cite{Stoilov1938}, 
a critical point should distort the foliation causing the fibers to
intersect. In the image of critical point the foliation is regular, 
as on figure~\ref{fig:foli-reg-sing}a, while 
in the critical point $p$ and its preimages there is an
intersection of $n$ fibers, where $n$ is the degree of $p$.
This situation is shown on figure~\ref{fig:foli-reg-sing}b.

\begin{figure}[htbp]
  \centering
  \includegraphics[scale=0.5]{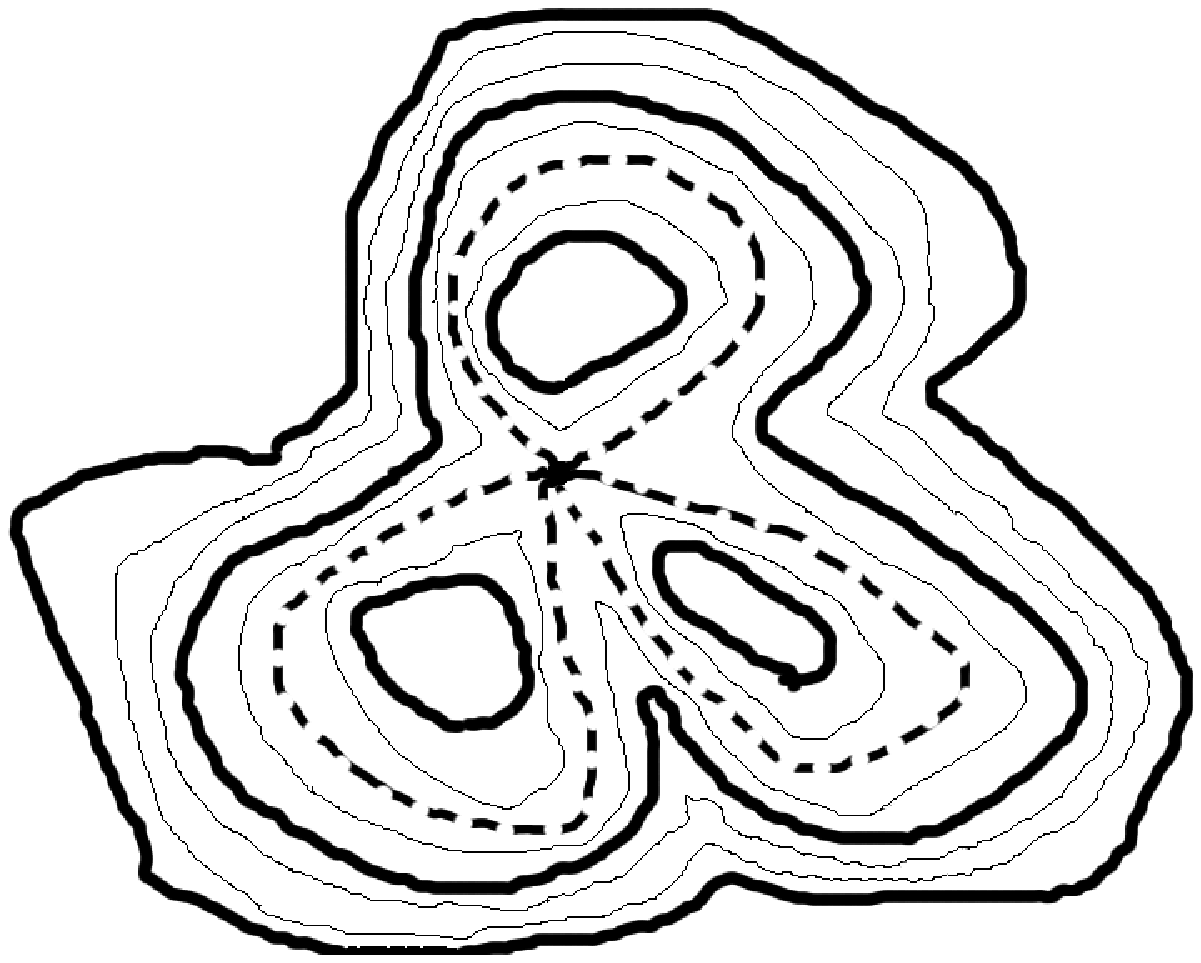}a)
  \includegraphics[scale=0.7]{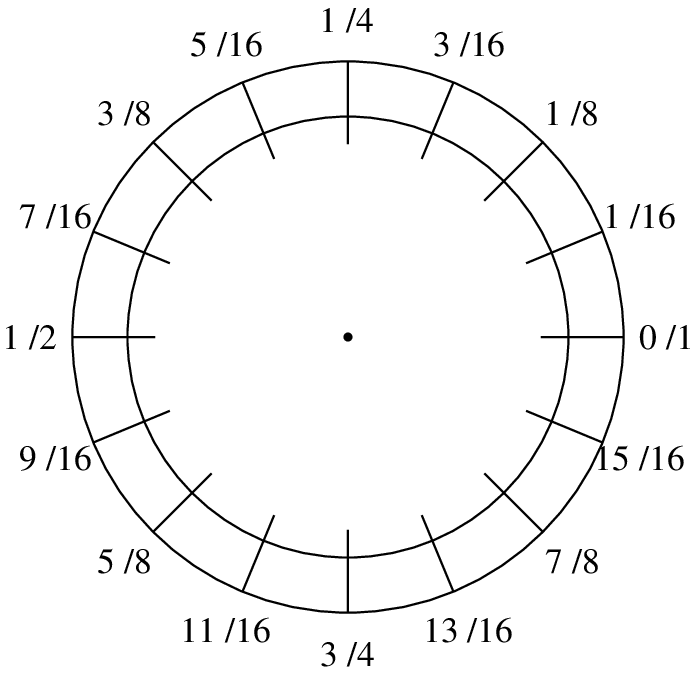}b)
  \caption{a) Wandering critical point.
    b) Numeration on $O^\perp(x)$.}
  \label{fig:fig3}
\end{figure}

Since fibers have dimension 1, they locally divide the surface $M$.
As shown in the proof of lemma~\ref{lm:fund_nei_is_ring}, the fibers
of the fundamental neighborhood $Q$ are circles.
Then, the critical preimage of a circle is a bouquet of circles.
Every circle of the bouquet divide the regular component $U$. 
Thus, preimage of $Q$ has $n+1$ components of boundary.
This situation is shown on the picture~\ref{fig:fig3}a).
The set $\cup_{k=-1}^{\infty} f^k(Q)$ also has $n+1$ components of
boundary because the rest of the fundamental neighborhoods $f^k(Q)$,
$k\ge 0$, are homeomorphic to the ring.
The further preimages of $Q$ are disjoint union of components of
connectedness each of them is either homeomorphic to $f^{-1}(Q)$ or cover it.
It means that $\cup_{k=-2}^{\infty} f^k(Q)$ will have not 
less than $n^2+1$ components of boundary, $\cup_{k=-3}^{\infty} f^k(Q)$
will have not less than $n^3+1$ components of boundary, and so on.
As a consequence, $U=\cup_{k=-\infty}^{\infty} f^k(Q)$ should have
infinitely many components of boundary.
If there are other critical points, the number of boundary
components just will grow faster, so the same reasoning is applicable
too. 
\end{proof}

\subsection{Coordinates on a fundamental neighborhood.}

Consider a saturated fundamental neighborhood $Q$ that is homeomorphic
to the ring as in lemma~\ref{lm:fund_nei_is_ring}.
Note that the boundary of $Q$ consists of foliation fibers.
Choose an orientation on $Q$.
As the points of $O^\perp(x)$ for a point $x\in Q$ belongs to the same
homeomorphic to a circle fiber (see remark~\ref{rem:single_fiber}), 
they are naturally cyclically ordered.
Assign a point $y\in f^{-n}\circ f^{n}(x)\subset O^\perp(x)$ a pair 
$\{k,m^{n}\}$, 
where $m$ is a degree of covering $Q$ by $f$, 
$n$ is the iteration number as in $f^{-n}\circ f^{n}$, 
the number $k$ is a cyclic order of $y$ among other points of the preimage
$(f^{-n}\circ f^{n})(x)$ according to the orientation of the fiber.
Note that $0 \le k < m^{n}$, $0$ is for $x$, $1$ is for the first another preimage of $x$
under $f^{-n}\circ f^{n}$ in the direction of orientation.
Let as write those pairs as fractions, for example, the pair
$\{k,m^{n}\}$ is written as fraction $\frac{k}{m^{n}}$.
Assign the point $x$ a pair $\frac{0}{1}$.
An example of the numeration for the degree $2$ is shown on the figure~\ref{fig:fig3}b).
Note that pairs that denote the same numeric fraction (like $\frac{1}{2}$ and
$\frac{2}{4}$) belong to the same point.

A fundamental neighborhood $Q$ already has one family of coordinate
lines due to the foliation. 
The other family can be built as follows. Choose a point
$x\in \partial Q$ such that $f(x)\in \partial Q$.
Then $x$ and $f(x)$ are on different connected components of $\partial Q$
which are fibers of the foliation. Connect $x$ and $f(x)$ with a
``transversal'' to the foliation jordan curve $\gamma$ that intersects
each fiber in a single point. Then the images $f^{-n}\circ
f^{n}(\gamma)$ will not intersect with each other because for every
$x\in\gamma$ the set $O^\perp(x)$ belongs to the fiber of $x$, but
$\gamma$ intersect fibers in a unique point.
Then the set $O^\perp(\gamma)=\cup_{n=0}^{\infty}f^{-n}\circ f^{n}(\gamma)$
will yield an everywhere dense family of curves
according to the definition~\ref{def:pseudolinear} of pseudolinearity.
Note that each image of the curve $\gamma$ can be assigned a pair
$\frac{k}{m^{n}}$ in the same way as for the images of a point.

This lamination $O^\perp(\gamma)$ can be extended to foliation by continuity.
For a point $x\in Q\setminus O^\perp(\gamma)$ define
$\mu(x)=\cap_{n\ge 0} \Delta_n(x)$, where $\Delta_n(x)$ is closure of
connected component of the set $Q\setminus f^{-n}\circ f^{n}(\gamma)$
that contains the point $x$. As an intersection of closed sets the set
$\mu(x)$ has non-empty intersection with every equiscalar line of
foliation of $Q$. 
Show that $\mu(x)$ intersect each fiber in a unique point.
Suppose on the contrary that there exists a fiber $\nu$ such that it is
intersected by $\mu(x)$ in 2 points $y_1$ and $y_2$. 
Denote by $z$ a point of intersection $\nu$ and $\gamma$.
Then by construction the whole segment $[y_1,y_2]$ of the fiber $\nu$
belongs to $\mu(x)$.  But it contradicts the fact that
$O^\perp(z)$ is everywhere dense in $\nu$ according to
definition~\ref{def:pseudolinear}.
Hence $\mu(x)$ intersects each equiscalar line in a unique point.
By construction, $\mu(x)$ continuously depends on equiscalar lines
foliation as $\mu(x)$ is majorated by images of $\gamma$.
As a consequence, it is also a jordan curve.

Thus, we constructed two families of curves on $Q$ such that their
fibers intersect each other in a single point. Note that the
foliation on equiscalar lines is a topological invariant of $f$,
while the second foliation is arbitrary up to choice of $\gamma$.

\subsection{Global Coordinates.}

Constructed above two families of curves on $Q$ generate
two families of curves on the whole regular component $U$
under the action of $f$ and $f^{-1}$. 
Let us call a foliation of $\Phi$ the neutral foliation 
and the foliation induced by $\gamma$ the timeline foliation.
By construction 
\begin{itemize}
\item they are invariant under the action of $f$;
\item they are regular except in critical points and their preimages.
\end{itemize}

\begin{remark}\label{rem:homeomorphism_on_timeline}
$f$ acts as homeomorphism on  non-critical fibers of the timeline
foliation. 
\end{remark}

\begin{proof}
  It follows from the fact that a fiber of timeline
  foliation intersects every fiber of neutral foliation, and, hence,
  intersects each set $O^\perp$ in no more then one point.
\end{proof}

Note that neutral coordinates on neutral fibers are defined on the
dense set $O^\perp(\gamma)$ and are
continuous by construction. 
They can be extended by continuity to a 
function $\nu\colon Q \to S^1= [0,1] \mod 1$.

By choosing a function $\tau'\colon\gamma\to[0,1]$ 
and extending it to constant on the neutral fibers
function $\tau'\colon Q\to[0,1]$ 
every point $p\in Q$ can be assigned unique coordinates
$(x=\nu(p),y=\tau(p))$, $x\in [0,1)$, $y\in [0,1]$. 

If there are no critical points those coordinates
can be extended to the coordinates on $U$
$(x,y)$, $x\in [0,1)$, $y\in\rr$ by the following rule. 
For a point $p\in U$ a neutral coordinate $x$ is determined by the
neutral coordinate of continuation of its timeline fiber, and
the timeline coordinate $y$ is determined by equation $y=y'+n$, 
where $y'=\tau(f^n(p))$ is timeline coordinate of $f^n(p)\in Q$.
Call them $\gamma$-coordinates.

\begin{lemma}\label{lm:pl_regcomp_no_critical_points_are_conjugate}
  Two regular components $U_1$ and $U_2$ of wandering sets of inner
  maps $f$ and $g$ of the same degree with a pseudolinear
  attracting isolated connected components of the boundary without critical points
  are topologically conjugate.
\end{lemma}

\begin{proof}
  Choose arbitrarily saturated fundamental neighborhoods 
  $Q_1$ and $Q_2$ of $U_1$ and $U_2$ and transversal curves 
  $\gamma_1\subset Q_1$ and $\gamma_2\subset Q_2$ to neutral foliations 
  in $Q_1$ and $Q_2$. Construction above extends them 
  to the two pairs of foliations in $Q_1$ and $Q_2$.

  Choose a homeomorphism $h'_\gamma\colon \gamma_1\to \gamma_2$.
  Continue $h'_\gamma$ over the timeline foliation using maps
  $(g^{-k}\circ g^{l})^{-1} \circ h'_\gamma\circ f^{-k}\circ f^{l}$, $k,l\ge 0$. 
  Note that branches of those maps are homeomorphisms when restricted
  on a single fiber (see remark~\ref{rem:homeomorphism_on_timeline})
  so they have inverse homeomorphisms which are denoted here as
  $(g^{-k}\circ g^{l})^{-1}$. Denote the obtained map by $\hat h_\gamma$.
  Since $f$ and $g$ preserve their neutral foliations, $\hat h_\gamma$
  maps points that belong to the same fiber of the neutral foliation
  of $f$ to the points on the same fiber of the neutral foliation of $g$.
  Because $f$ and $g$ have the same map degree they act the same way
  on the points that have the same neutral coordinates.
  It means that $\hat h_\gamma$ induces not only mapping of timeline
  foliations, but also a mapping of neutral foliations. 

  It can be extended by continuity to a bijection $h_\gamma$ between $U_1$ and $U_2$.
  This bijection is continuous, since foliations are continuous.
  Also $h_\gamma$ is an open map because in every point it maps its base of topology
  built from foliation boxes to the corresponding base of topology
  generated by foliations in the image.
  It shows that $h_\gamma$ is a homeomorphism.
  By construction, $h_\gamma\circ f = g\circ h_\gamma$, it means that
  $f$ and $g$ are topologically conjugate.

  Note that if one choose $h'_\gamma$ to map identical timeline
  $\gamma$-coordinates  on $U_1$ and $U_2$ then $h_\gamma$ 
  in corresponding $\gamma$-coordinates on $U_1$ and $U_2$
  will be nothing but $\Id$.
\end{proof}

\section{Distinguishing graph.}\label{sec:distinguishing_graph}

As shown in lemma~\ref{lm:pl_regcomp_no_critical_points_are_conjugate}
in case when there are no critical points the topological
classification is simple. The only essential invariant is the map degree.

In case when there are critical points the symmetry of points breaks.
Topological conjugacy distinguish among critical and non-critical
points, their preimages and foliation fibers.
To deal with this extra topological invariants we need to build 
a distinguishing graph which is technically a compact encoding of some
labelled Konrod-Reeb graph of a specially chosen part of the foliation
function. 

Choose the fundamental neighborhood to be the neighborhood between the
boundary of the first and second images of the first critical level of
the invariant neutral foliation.
Note that due to the choice of the fundamental neighborhood 
and the genericity condition from
definition~\ref{def:genericity_barticle4} 
any fiber of the neutral foliation in the chosen fundamental neighborhood
has no more than one image of a critical point
and there are finitely many fibers that have an image of a critical point.
Choose the line $\gamma$ that generates the timeline foliation to
intersect each fiber with a critical point in that critical point,
including the boundary fibers of the fundamental neighborhood. 
In absence of critical points the points of regular component can 
be uniquely identified by a pair of coordinates on its timeline and
neutral fibers ($\gamma$-coordinates introduced above).
But the critical points make foliation singular so that lines of
timeline foliation intersect in critical points and their preimages.
Also preimages of the fundamental neighborhood divide onto components
such that internals of those components are mutually disconnected.
It follows that timeline $\gamma$-coordinates can be defined the same
way as in absence of critical points, but neutral $\gamma$-coordinates
can't. 

In that case introduce neutral $\gamma$-coordinates that are local to a
component of a preimage of the fundamental neighborhood $Q$.
Consider the fiber $\widehat\gamma$ of timeline foliation that contains the original 
curve $\gamma\in Q$. This fiber branches in critical points and their
preimages simultaneously together with dividing of preimages of 
the fundamental neighborhood $Q$ onto components so than there is
a unique branch of the singular fiber $\widehat\gamma$ that goes into
internals of each component. Now construct neutral coordinates for
every component using this branch of $\widehat\gamma$ as local origin 
$\frac{0}{1}$. At that points at the component boundary might get 
different neutral coordinates from components the point borders with.
To avoid the ambiguity let us always assume that the component
of the boundary further from attractor should belong to the component
and inherit its numeration. From that the critical points and their
preimages will get the neutral coordinate $\frac{0}{1}$.

Also, those coordinates are not unique: there are finitely many
components having the same coordinates.
To avoid this ambiguity component of a preimage should be counted for
every $n$ such that every point of $U$ could be identified 
with a triplet $(x,y,C)$, where $x$ is a local neutral coordinate,
$y$ is a global timeline coordinate, and $C$ is a component number.

To enumerate the components 
choose an orientation on the boundary of the original fundamental neighborhood $Q$.
Since the regular component is oriented, this orientation induces an
orientation on the preimages of the boundary.
Starting from a point on the original curve $\gamma$ and move on
the singular fiber $\widehat\gamma$ there is a unique shortest path to 
walk over all components of preimage of the border of $Q$ moving
on the preimages according to the orientation and moving according to
the orientation from one component to another on $\widehat\gamma$.
That path cyclically orders components. Number them according to that
order so the first component met on that path will get the number 1.
This numeration is related to the first critical point.

The other critical points further subdivide some of those components,
in such a way that those subcomponents share the same regular
component of the singular neutral foliation of the previous critical 
point. so in that case it is natural to create a compound number 
by assigning such a component the sequence of numbers each related to
the critical neutral fiber of corresponding critical point.
Call the obtained compound number $C=\{k_1,\dots,k_l\}$ a component number.

By construction critical points are identified by a triplet
$(\frac{0}{1},n+y',C)$, which does not depend on the choice of $\gamma$.
It only has compound component number depending on choice of orientation
and timeline coordinate $y$ having invariant integer part $n$ and a
fraction part $y'$ depending on the choice of the timeline
function $\tau$. However, the relative order of images and preimages
of critical neutral fibers does not depend on the choice of $\tau$.

Project the critical points on the semi-interval $[0,1)$ using their 
fractional parts $y'$ of their timeline coordinates and label 
them with pairs $(d,n,C \overline C)$, where $d$ is a local degree of
the critical point,  $C \overline C$ is an
unordered pair of component numbers with different choices of
orientation in $Q$.
Note that the first critical point gets by construction the label
$(...,0,\{1\}\{1\})$.

\begin{definition}
  Two labels are equal if their components are equal:
  the first and second components are equal as numbers, 
  the third components are equal as unordered pairs of vectors.
\end{definition}

\begin{definition}
  Call the semi-interval $[0,1)$ either without labelled points 
  or with labelled points such that $0$ is labelled and $0$'s label
  looks like $(...,0,\{1\}\{1\})$ a distinguishing graph.
\end{definition}

\begin{definition}
  Two distinguishing graphs are equivalent if there exists 
  a homeomorphism of the semi-interval $[0,1)$ such that the labels 
  of images and preimages are equal.
\end{definition}

\begin{proof}[Proof of the main theorem.]
The proof is similar to the proof of the
lemma~\ref{lm:pl_regcomp_no_critical_points_are_conjugate},
except we have not global but local coordinate families in 
each component of preimage of the fundamental neighborhood
that are built is section~\ref{sec:distinguishing_graph}.
By definition the equivalence of the distinguishing graphs guarantees 
one-to-one correspondence between adjacent components.
As in lemma ~\ref{lm:pl_regcomp_no_critical_points_are_conjugate},
it follows that a map written in that coordinates as $\Id$
is the desired homeomorphism.
\end{proof}

\end{document}